\documentclass[11pt,a4paper]{article}
\usepackage[T1]{fontenc}
\usepackage{lmodern}
\usepackage[margin=26mm]{geometry}
\usepackage{amsmath,amssymb,amsthm,mathtools}
\usepackage{microtype}
\usepackage{enumitem}
\usepackage{needspace}
\usepackage{xcolor}
\usepackage[colorlinks=true,linkcolor=blue!45!black,citecolor=blue!45!black,urlcolor=blue!45!black]{hyperref}
\hypersetup{pdftitle={An Improved Bound on Spanning Bipartite Connectivity},pdfauthor={G. Gutin, Y. Hao and Y. Zhou}}
\newtheorem{theorem}{Theorem}[section]
\newtheorem{lemma}[theorem]{Lemma}
\newtheorem{claim}[theorem]{Claim}
\newtheorem{proposition}[theorem]{Proposition}
\newtheorem{obs}[theorem]{Observation}

\theoremstyle{remark}

\newcommand{\E}{\mathbb E}

\setlist[enumerate]{itemsep=2pt,topsep=4pt}
\setlist[itemize]{itemsep=2pt,topsep=4pt}
\allowdisplaybreaks[1]
\title{An $O(k\log(n/k))$ Bound on Spanning Bipartite Connectivity}
\author{G. Gutin\footnote{Royal Holloway University of London, UK}, Y. Hao\footnote{Nankai University, PR China} and Y. Zhou\footnote{Shenzhen Institutes of Advance Technology, PR China}}
\date{}

\begin{document}
	\maketitle
	\begin{abstract}
		For integers $1\le k\le n/2$, let $f(k,n)$ be the least integer $s$ such that every $s$-connected graph on $n$ vertices contains a spanning bipartite $k$-connected subgraph. Thomassen conjectured that $f(k,n)$ is bounded by a function of $k$ alone. Delcourt and Ferber proved $f(k,n)=O(k^3\log n)$, and Yuster subsequently obtained $f(k,n)\le22k^2\log_2 n$. We prove that, for $2\le k\le n/2$,
		\[
		f(k,n)\le\min\left\{n-1,\,
		\left\lfloor6(k-1)\log_2\frac{n}{k-1}\right\rfloor\right\}.
		\]
		In particular, $f(k,n)=O(k\log(n/k))$. 
	\end{abstract}
	
	
	\section{Introduction}
	All graphs considered here are finite and simple. A graph $H$ is \emph{$k$-connected} if $|V(H)|\ge k+1$ and $H-X$ is connected whenever $X\subseteq V(H)$ and $|X|\le k-1$. For integers $1\le k\le n/2$, define
	\[
	f(k,n)=\min\left\{s\in\mathbb Z_{\ge0}:
	\begin{array}{l}
		\text{every $s$-connected graph on $n$ vertices contains}\\[-1pt]
		\text{a spanning bipartite $k$-connected subgraph}
	\end{array}\right\}.
	\]
	Thomassen conjectured in 1989 that $f(k,n)$ is bounded by a function of $k$ alone~\cite{Thomassen}. Delcourt and Ferber obtained $f(k,n)=O(k^3\log n)$~\cite{DF}. Yuster subsequently proved
	\[
	f(k,n)\le22k^2\log_2 n,
	\]
	together with stronger estimates in certain ranges where $k$ grows with $n$~\cite{Yuster}. For $k=1$, a spanning tree gives $f(1,n)=1$. For $k=2$, the exact value is $f(2,n)=3$ for every $n\ge5$~\cite{Yuster}.
	
	Our main result improves the dependence on $k$ in the general logarithmic estimate from quadratic to linear and expresses the logarithmic term in terms of $n/k$.
	
	\begin{theorem}\label{thm:main}
		For all integers $2\le k\le n/2$,
		\begin{equation}\label{eq:main}
			f(k,n)\le\min\left\{n-1,\,
			\left\lfloor6(k-1)\log_2\frac{n}{k-1}\right\rfloor\right\}.
		\end{equation}
	\end{theorem}
	
	The constant in~\eqref{eq:main} is chosen for a simple statement. Theorem~\ref{thm:quantitative} gives a more precise threshold using the same proof.
	
	The proof follows the minimum-partition approach used by Yuster. Fix an integer $d\ge2k$ and partition the vertex set into as few parts as possible, each a singleton or a set of at least $d$ vertices supporting a spanning bipartite $k$-connected graph. We call such a partition admissible. High connectivity supplies a boundary matching for every non-singleton part. For each singleton we use all its incident edges. We refer to these selected edges as \emph{attachment edges}.
	
	Independently reversing the bipartitions of the parts gives a random global bipartition. Minimality bounds each of the two compatibility classes of attachment edges toward a fixed target part by $k-1$. This yields the exponential moment estimate
	\[
	\E\,2^{-|F_i|/(k-1)}
	\le (3/4)^{|S_i|/(k-1)},
	\]
	where $S_i$ is the initial attachment set of a part and $F_i$ is its bichromatic subset. Summing these moments with suitable weights gives a colouring for which a single potential is less than one.
	
	The extraction step alternates two operations. First, a singleton of current degree less than $k$ is pruned from the proposed merger. Its removal decreases the potential. Once all such singletons have been pruned, the current graph has minimum degree at least $k$. If it is not $k$-connected, we pass to the parts represented in a component behind a separator of order at most $k-1$. A component carrying at most half the potential can be chosen, and each retained attachment set loses at most $k-1$ edges. The resulting factor of at most two exactly compensates for the halving. A bipartite size estimate keeps the proposed merger above the admissibility threshold $d$ throughout the process.
	
	
	\section{Preliminaries}\label{sec:prelim}
	We denote by $d_H(v)$, $\delta(H)$ and $\kappa(H)$ the degree of the vertex $v$, the minimum degree and the vertex-connectivity of a graph $H$, respectively. For disjoint vertex sets $A,B\subseteq V(G)$, let
	\[
	E_G(A,B)=\{ab\in E(G):a\in A,\ b\in B\}.
	\]
	We omit the subscript when the ambient graph is clear. Let  $[t]=\{1,\ldots,t\}$. 
	
	The elementary bound $f(k,n)\le n-1$ follows because the only $(n-1)$-connected graph on $n$ vertices is $K_n$, which contains a spanning $K_{k,n-k}$. As $k\leq n/2$, both sides of this complete bipartite graph have at least $k$ vertices, so it is $k$-connected. We end this section by introducing the following two very simple merging operations for obtaining larger $k$-connected bipartite subgraphs.
	
	\begin{proposition}\label{prop:merge}
		Let $H_1,H_2$ be vertex-disjoint bipartite $k$-connected subgraphs of $G$, with bipartitions $(A_1,B_1)$ and $(A_2,B_2)$. If there is a matching of size at least $k$ in $E_G(A_1,B_2)\cup E_G(B_1, A_2)$ or $E_G(A_1,A_2)\cup E_G(B_1,B_2)$, then $G[V(H_1)\cup V(H_2)]$ has a spanning bipartite $k$-connected subgraph. 
	\end{proposition}
	\begin{proof}
		Suppose without loss of generality that $E_G(A_1,B_2)\cup E_G(B_1, A_2)$ contains a matching of at least $k$ edges.  Deleting at most $k-1$ vertices leaves each $H_i=(A_i,B_i)$ nonempty and connected. In addition, at least one of the chosen matching edges in $E_G(A_1,B_2)\cup E_G(B_1, A_2)$ survives. Thus, the union of $H_1$, $H_2$ and the chosen matching edges is a bipartite $k$-connected graph with bipartition $(A_1\cup A_2,B_1\cup B_2)$. 
	\end{proof}
	
	\begin{proposition}\label{prop:absorb}
		Let $H=(A,B)$ be a bipartite $k$-connected subgraph of $G$ and $v\notin V(H)$. If $v$ has $k$ neighbours in $A$ or $B$, then $G[V(H)\cup\{v\}]$ has a spanning bipartite $k$-connected subgraph. 
	\end{proposition}
	\begin{proof}
		Suppose without loss of generality that $v$ has $k$ neighbours in $B$. Consider the bipartite subgraph $(A\cup \{v\}, B)$ of $G$ formed by $H$ and $k$ edges from $v$ to $B$. As $H$ is $k$-connected, after deletion of at most $k-1$ vertices, $H$ remains nonempty and connected. In addition, if $v$ survives, so does at least one of its chosen neighbours. Thus, the bipartite subgraph  $(A\cup \{v\}, B)$ is $k$-connected. 
	\end{proof}
	
	\section{The minimum admissible partition and a potential function}\label{sec:partition}
	Let $k\geq 2$ and fix an integer $d\ge2k$. A partition
	\[
	\mathcal P=\{V_1,\ldots,V_t\}
	\]
	of $V(G)$ is admissible w.r.t. $d$ if each part is a singleton or has at least $d$ vertices and admits a spanning bipartite $k$-connected graph. Choose an admissible partition with the minimum number of parts. Such a partition exists because the partition into singletons is always admissible. For each non-singleton part $V_i$, fix a spanning bipartite $k$-connected graph $H_i$ on $V_i$ and a proper bipartition $(A_i,B_i)$ of $H_i$. For a singleton part $V_i=\{v_i\}$, let $H_i$ be the isolated graph on its vertex, and set $A_i=\varnothing$ and $B_i=\{v_i\}$.
	
	\begin{lemma}\label{lem:boundary}
		If $t\geq 2$ and $G$ is $(d+2k-3)$-connected, then every non-singleton part $V_i$ has a matching of size at least $d$ in $E_G(V_i,V(G)\setminus V_i)$.
	\end{lemma}
	\begin{proof}
		Let $\overline{V_i}=V(G)\setminus V_i$. Suppose the largest boundary matching has size $a<d$. By K\"onig's theorem, the bipartite graph with vertex classes $V_i,\overline{V_i}$ and edge set $E_G(V_i,\overline{V_i})$ has a vertex cover $C$ of size $a$. Since $|V_i|\ge d>a$, the set $V_i\setminus C$ is nonempty. 
		
		If $\overline{V_i} \setminus C$ is also nonempty, deleting $C$ would disconnect $G$, contrary to its $(d+2k-3)$-connectivity and the fact that $d+2k-3\geq d+1>a=|C|$. Hence $\overline{V_i}\subseteq C$, and $1\le|\overline{V_i}|\le a\le d-1$. Every part contained in $\overline{V_i}$ is therefore a singleton. Let $V_j=\{v\}$ be one of these singletons. Then, $v$ has at most $2k-2$ neighbours in $V_i$. In fact, if not, then by the pigeonhole principle $v$ has at least $k$ neighbours in $A_i$ or $B_i$ and therefore by Proposition~\ref{prop:absorb} $(\{V_1,\ldots,V_t\}\setminus \{V_i,V_j\})\cup \{V_i\cup V_j\}$ is an admissible partition with fewer parts, a contradiction to the  minimality of $\mathcal P$. Thus,
		\[
		d_G(v)\le2k-2+|\overline{V_i}|-1\le d+2k-4,
		\]
		which contradicts $\delta(G)\ge\kappa(G)\geq d+2k-3$. This completes the proof.
	\end{proof}
	
	Suppose now that $t\ge2$ and $G$ is $s$-connected, where $s\ge d+2k-3$.
	For each $i\in[t]$, choose an edge set $S_i$ as follows:
	\begin{itemize}
		\item if $V_i$ is non-singleton, choose exactly $d$ edges of a boundary matching supplied by Lemma~\ref{lem:boundary};
		\item if $V_i=\{v_i\}$, let $S_i$ be the set of all edges incident with $v_i$.
	\end{itemize}
	Thus
	\begin{equation}\label{eq:initial-sizes}
		|S_i|=d\quad\text{for non-singleton parts},\qquad
		|S_i|=d_G(v_i)\ge s\quad\text{for singleton parts}.
	\end{equation}
	
	For distinct $i,j\in[t]$, let $p_{ij}$ and $q_{ij}$ denote the numbers of edges in $S_i$ in the following two classes:
	\[
	\begin{aligned}
		p_{ij}&=|S_i\cap (E_G(A_i,B_j)\cup E_G(B_i,A_j))|,\\
		q_{ij}&=|S_i\cap (E_G(A_i,A_j)\cup E_G(B_i,B_j))|.
	\end{aligned}
	\]
	In particular, if $V_i=\{v_i\}$ is a singleton, then $A_i=\varnothing$ and $B_i=\{v_i\}$, so $p_{ij}$ counts the edges from $v_i$ to $A_j$, and $q_{ij}$ counts those from $v_i$ to $B_j$. Since $S_i$ contains all edges incident with $v_i$, these are precisely the numbers of its neighbours in $A_j$ and $B_j$, respectively.
	
	Then, the following holds. 
	
	\begin{lemma}\label{lem:compatibility}
		If $t\geq 2$, then for distinct $i,j\in[t]$,
		\[
		0\le p_{ij}\le k-1,\qquad 0\le q_{ij}\le k-1.
		\]
	\end{lemma}
	\begin{proof}
		Suppose for a contradiction that $p_{ij}\geq k$ (the proof for $q_{ij}$ is identical). If $V_j$ were a singleton, $S_i$ would contain at most one edge incident with $V_j$, since $S_i$ is a matching or a simple star. This would contradict $p_{ij}\ge k\ge2$. Hence we consider two cases based on whether $V_i$ is a singleton.
		
		If $V_i$ is non-singleton, then the $p_{ij}$ edges form a matching in $E_G(A_i,B_j)\cup E_G(B_i,A_j)$ of size at least $k$ and therefore by Proposition~\ref{prop:merge} we can merge $V_i$ and $V_j$, contradicting minimality of $\mathcal P$. If $V_i=\{v\}$ is a singleton, then $v$ has at least $k$ neighbours in $A_j$ and therefore by Proposition~\ref{prop:absorb} we can merge $V_i$ and $V_j$, again contradicting minimality of $\mathcal P$. This completes the proof.x
	\end{proof}
	
	For each non-singleton part, independently choose one of the two coloring ($A_i$ blue and $B_i$ red or $B_i$ blue and $A_i$ red) of its bipartition $(A_i,B_i)$ with probability $1/2$, and independently colour each singleton red or blue with equal probability. This gives a red--blue colouring of all vertices, proper on each $H_i$. For each $i\in[t]$, let
	\[
	F_i=\{xy\in S_i:x\text{ and }y\text{ have different colours}\}.
	\]
	
	For a fixed outcome of this colouring and $W\subseteq[t]$, put
	\[
	U(W)=\bigcup_{i\in W}V_i,
	\qquad F_i(W)=F_i\cap E(V_i,U(W)\setminus V_i)\quad(i\in W),
	\]
	and let $J(W)$ be the simple graph on $U(W)$ formed by the graphs $H_i$ and the edge sets $F_i(W)$ with $i\in W$. Then $J(W)$ is bipartite, and $F_i(W)$ is a matching whenever $V_i$ is non-singleton. If $V_i=\{v_i\}$, $F_i(W)$ is exactly the edges in $J(w)$ incident with $v_i$, and therefore 
	\begin{equation}\label{eq:full-star}
		|F_i(W)|=d_{J(W)}(v_i).
	\end{equation}
	
	Define the singleton weight
	\begin{equation}\label{eq:omega}
		\omega=\max\left\{
		2(1-2^{-\frac{1}{k-1}}),\,
		\frac{2^{\frac{k+1}{2(k-1)}}}{k+1},\,
		\frac{2^{\frac{d-1}{2(k-1)}}}{d-1}
		\right\}.
	\end{equation}
	Give each non-singleton part weight $w_i=1$ and each singleton part weight $w_i=\omega$ and define the potential function as
	\begin{equation}\label{eq:potential}
		\Phi(W)=\sum_{i\in W}w_i2^{-\frac{|F_i(W)|}{k-1}}
		=\sum_{\substack{i\in W\\|V_i|>1}}2^{-\frac{|F_i(W)|}{k-1}}
		+\omega\sum_{\substack{i\in W\\|V_i|=1}}2^{-\frac{|F_i(W)|}{k-1}}.
	\end{equation}
	
	One can observe from \eqref{eq:potential} that if $W=\{i\}$ and $V_i$ is non-singleton, then $|F_i(W)|=0$ and therefore
	$\Phi(W)=2^{0}=1$. Thus, we have the following observation. 
	
	\begin{obs}\label{obs:phi=1}
		If	$W=\{i\}$ and $V_i$ is non-singleton, then $\Phi(W)=1$. 
	\end{obs}

	\begin{lemma}\label{lem:barrier}
		If $Q$ is a bipartite graph on $z$ vertices, where $k+1\le z\le d-1$, then
		\[\omega\sum_{v\in V(Q)}2^{-\frac{d_Q(v)}{k-1}}\ge1.\]
	\end{lemma}
	\begin{proof}
		By Mantel's Theorem, we have $e(Q)\le z^2/4$. Combining it with Jensen's inequality, we obtain
		\[
		\sum_{v\in V(Q)}2^{-\frac{d_Q(v)}{k-1}}
		\ge z2^{\frac{-2e(Q)}{z(k-1)}}\ge z2^{-\frac{z}{2(k-1)}}.
		\]
		As the function $\ln(z2^{-\frac{z}{2(k-1)}})=\ln z-\frac{\ln2}{2(k-1)}z$ is concave on the positive real axis, its minimum on $[k+1,d-1]$ occurs at an endpoint. The last two terms in~\eqref{eq:omega} ensure that $\omega z2^{-\frac{z}{2(k-1)}}\ge1$ at both endpoints, and therefore throughout this interval.
	\end{proof}
	
	\begin{lemma}\label{lem:component-size}
		Let $Q$ be bipartite with $\delta(Q)\ge k$, and let $S\subseteq V(Q)$ with $|S|\le k-1$. Every component $C$ of $Q-S$ has at least $2k-|S|\ge k+1$ vertices.
	\end{lemma}
	\begin{proof}
		Let $(A,B)$ be the bipartition of $Q$. The component $C$ meets both $A$ and $B$: otherwise every neighbour of each of its vertices would lie in $S$, contradicting minimum degree at least $k$. Choose $a\in A\cap V(C)$ and $b\in B\cap V(C)$. All neighbours of $a$ lie in $(B\cap V(C))\cup(B\cap S)$, and all neighbours of $b$ lie in $(A\cap V(C))\cup(A\cap S)$. Adding the two degree bounds yields $2k\le |V(C)|+|S|$, as desired. 
	\end{proof}
	
	\begin{lemma}\label{lem:extraction}
		Suppose that $|U([t])|\ge d$ and $\Phi([t])<1$. Then there is a set $W\subseteq[t]$ with $|W|\ge2$ such that $|U(W)|\ge d$ and $J(W)$ is bipartite and $k$-connected.
	\end{lemma}
	\begin{proof}
		Let $W\subseteq[t]$ have minimum cardinality subject to the following conditions. 
		\begin{equation}\label{eq:invariants}
			|U(W)|\ge d,\qquad \Phi(W)<1.
		\end{equation}
		Note that as $|U(W)|\ge d>1$, if $W=\{i\}$ for some $i\in [t]$ then $V_i$ must be non-singleton and therefore by Observation \ref{obs:phi=1}, we must have $ \Phi(W)=1$, a contradiction. Thus, $|W|\geq 2$ and it remains to show that $J(W)$ is $k$-connected. Suppose that $J(W)$ is not. The following claim now holds.
		
		\begin{claim}\label{cl:mdk}
			$\delta(J(W))\geq k$.
		\end{claim}
		\begin{proof}
			Let $v\in V(J(W))$ be arbitrary and $V_i$ be the part containing it. If $V_i$ is a non-singleton part then as $H_i\subseteq J(W)$ is $k$-connected, $d_{J(W)}(v)\geq d_{H_i}(v)\geq k$. Suppose that $V_i$ is a singleton and $d_{J(W)}(v)\leq k-1$. We now reach a contradiction by showing that $W\setminus\{i\}$ also satisfies~\eqref{eq:invariants}.
			
			We first show that $\Phi(W\setminus \{i\})< 1$. Indeed, by~\eqref{eq:full-star} and \eqref{eq:omega}, the contribution of $i$ is at least
			\[
			\omega 2^{-\frac{k-1}{k-1}}\geq 2(1-2^{-\frac{1}{k-1}})\cdot 2^{-1}= 1-2^{-\frac{1}{k-1}}.
			\]
			Deleting $v$ removes at most one edge from any other set $F_j(W)$. Thus,
			\begin{align*}
				\Phi(W\setminus\{i\})
				&\le 2^{\frac{1}{k-1}}\bigl(\Phi(W)-(1-2^{-\frac{1}{k-1}})\bigr)\\
				&=\Phi(W)+\frac{1-2^{-\frac{1}{k-1}}}{2^{-\frac{1}{k-1}}}\bigl(\Phi(W)-1\bigr)
				<\Phi(W) <1.
			\end{align*}
			
			It remains to show that $|U(W\setminus\{i\})|\ge d$. Suppose otherwise. Since $|U(W)|\ge d$ and $V_i=\{v\}$, we have $|U(W\setminus\{i\})|=d-1$. Every remaining part is therefore a singleton. As $d-1\ge k+1$, Lemma~\ref{lem:barrier} and~\eqref{eq:full-star} give
			\[
			\Phi(W\setminus\{i\})
			=\omega\sum_{u\in U(W\setminus\{i\})}
			2^{-\frac{d_{J(W\setminus\{i\})}(u)}{k-1}}\ge1,
			\]
			a contradiction. Thus $W\setminus\{i\}$ satisfies~\eqref{eq:invariants}, contrary to the minimality of $W$. This completes the proof of the claim.
		\end{proof}
		
		Since $J(W)$ is not $k$-connected, there is a set $S\subseteq U(W)$, $|S|\le k-1$, such that $J(W)-S$ has at least two components. For each non-singleton part $V_i$, the graph $H_i-(V_i\cap S)$ is nonempty and connected. Assign to a component $C$ the mass
		\[
		\sum_{\substack{i\in W:\ V_i\cap V(C)\neq\varnothing}}
		w_i2^{-\frac{|F_i(W)|}{k-1}}.
		\]
		The component masses sum to at most $\Phi(W)$. Since there are at least two components, choose one with mass at most $\Phi(W)/2$, and let $W'$ be the indices with $V_i\cap V(C)\neq\varnothing$. Since every other component meets a part not in $W'$, we have $\varnothing\ne W'\subsetneq W$.
		
		Fix $i\in W'$. Clearly, any edge of $F_i(W)\setminus F_i(W')$ must be incident with $S$. Thus, by the definition of $F_i(W)$ we have that
		\[
		|F_i(W')|\ge |F_i(W)|-|S|\ge |F_i(W)|-(k-1),
		\]
		and therefore
		\[
		\Phi(W')=\sum_{i\in W'}w_i2^{-\frac{|F_i(W')|}{k-1}}\le \sum_{i\in W'}w_i2^{-\frac{|F_i(W)|-(k-1)}{k-1}}
		\le2\cdot\frac{\Phi(W)}2=\Phi(W)<1.
		\]
		
		By Claim~\ref{cl:mdk} and Lemma~\ref{lem:component-size}, the chosen component \(C\) has at least \(k+1\) vertices. Since \(V(C)\subseteq U(W')\), we have
		\[
		|U(W')|\ge k+1.
		\]
		We claim that in fact \(|U(W')|\ge d\). Suppose otherwise. Since every non-singleton part has at least \(d\) vertices, all parts indexed by \(W'\) must be singletons. Hence \(U(W')=V(C)\), and therefore
		\[
		k+1\le |U(W')|\le d-1.
		\]
		However, Lemma~\ref{lem:barrier} now gives \(\Phi(W')\ge1\), contradicting \(\Phi(W')<1\). Thus \(|U(W')|\ge d\), so \(W'\) also satisfies~\eqref{eq:invariants}, contrary to the minimality of \(W\).
	\end{proof}
	
	\section{Proof of the main theorem}\label{sec:proof}
	We first record the quantitative consequence of the preceding lemmas.
	
	\begin{theorem}\label{thm:quantitative}
		Let $2\le k\le n/2$ and let $d\ge2k$ be any integer satisfying
		\begin{equation}\label{eq:d-condition}
			d(4/3)^{\frac{d}{k-1}}>n.
		\end{equation}
		Define $\omega$ by~\eqref{eq:omega}, and
		\begin{equation}\label{eq:s-threshold}
			s=\max\left\{d+2k-3,\,
			\left\lfloor (k-1)\log_{4/3}(n\omega)\right\rfloor+1\right\}.
		\end{equation}
		Then $f(k,n)\le\min\{n-1,s\}$.
	\end{theorem}
	\begin{proof}
		The bound $n-1$ has been explained in Section~\ref{sec:prelim}, so assume $s<n-1$ and let $G$ be an $s$-connected graph on $n$ vertices. In particular, $n>d$. Choose a minimum admissible partition w.r.t. $d$ as in Section~\ref{sec:partition}. If it has one part, it already provides the required spanning bipartite $k$-connected graph. Suppose it has $t\ge2$ parts.
		
		Choose the edge sets $S_i$ and consider the random two-colouring $c$ defined after the proof of Lemma~\ref{lem:compatibility}. Denote the corresponding edge sets by $F_i^c$. The potential $\Phi([t])$ is then a random variable depending on $c$.
		
		Fix $i\in[t]$ and condition on the colouring of $V_i$. For each $j\ne i$, put
		\[
		X_{ij}(c)=|F_i^c\cap E_G(V_i,V_j)|.
		\]
		Then $X_{ij}(c)$ takes the values $p_{ij}$ and $q_{ij}$ with equal probability, and these random variables are conditionally independent as $j$ varies. By Lemma~\ref{lem:compatibility}, both values lie in $[0,k-1]$. The following claim now holds.
		
		\begin{claim}\label{cl:phi}
			$\E \left[\Phi([t])\right]<1$. 
		\end{claim}
		
		\begin{proof}
			
			For every $i\in[t]$, conditional independence gives
			\begin{equation}\label{eq:moment}
				\E\left[2^{-\frac{|F_i^c|}{k-1}}\,\middle|\,c|_{V_i}\right]
				=\prod_{j\neq i}\E \left[2^{-\frac{X_{ij}(c)}{k-1}}\,\middle|\,c|_{V_i}\right]=\prod_{j\ne i}\frac{2^{-\frac{p_{ij}}{k-1}}+2^{-\frac{q_{ij}}{k-1}}}{2}
				\le\prod_{j\ne i}(3/4)^{\frac{p_{ij}+q_{ij}}{k-1}}
				=(3/4)^{\frac{|S_i|}{k-1}},
			\end{equation}
			where $c|_{V_i}$ denotes the colouring restricted to $V_i$. The product does not depend on this colouring, so the same estimate holds without conditioning. The inequality follows by taking $x=p_{ij}/(k-1)$ and $y=q_{ij}/(k-1)$ and observing that the function $\frac{2^{-x}+2^{-y}}2(4/3)^{x+y}$ is convex on $[0,1]^2$. Its maximum is therefore attained at a corner point. Its values at the four corner points are $1,1,1,8/9$, so the function is at most one throughout the square.
			
			Let $a$ be the number of non-singleton parts and $b$ the number of singleton parts. Then $ad+b\le n$. By \eqref{eq:moment} and~\eqref{eq:initial-sizes}, we have
			
			\[
			\E\left[\Phi([t])\right]
			\le a(3/4)^{\frac{d}{k-1}}+\omega b(3/4)^{\frac{s}{k-1}}
			<\frac{ad+b}{n}\le1.
			\]
			Here the non-singleton estimate follows from~\eqref{eq:d-condition}, and the singleton estimate follows from
			\[
			s>(k-1)\log_{4/3}(n\omega),
			\]
			completing the proof of the claim.
		\end{proof}
		
		Choose a colouring with $\Phi([t])<1$. Lemma~\ref{lem:extraction} now gives a union of at least two parts, of order at least $d$, supporting a spanning bipartite $k$-connected graph. Replacing those parts by their union produces an admissible partition with fewer parts, contradicting minimality. Hence $G$ has the required spanning subgraph.
	\end{proof}
	
	We are now ready to prove our main result.
	
	\begin{proof}[Proof of Theorem~\ref{thm:main}]
		Choose
		\begin{equation}\label{eq:d-explicit}
			d=\left\lfloor(k-1)\log_{4/3}\frac{n}{k-1}\right\rfloor+1.
		\end{equation}
		To check that $d\ge2k$, put $y=2k/(k-1)\in(2,4]$. The function $\ln y-y\ln(4/3)$ is concave and positive at both endpoints of $[2,4]$. Hence $\log_{4/3}y\ge y$. Since $n/(k-1)\ge y$,~\eqref{eq:d-explicit} gives $d>2k$. Also $(4/3)^{d/(k-1)}>n/(k-1)$ and $d\geq 2k>k-1$, so~\eqref{eq:d-condition} holds.
		
		We next claim that
		\begin{equation}\label{eq:weight-bound}
			(k-1)\omega\le 2^{\frac12\log_{4/3}\frac{n}{k-1}}.
		\end{equation}
		Since $n/(k-1)>2$, the right-hand side is greater than $\sqrt2$. For the first term in~\eqref{eq:omega},
		\[
		2(k-1)(1-2^{-1/(k-1)})\le2\ln2<\sqrt2.
		\]
		For the second, put $v=(k+1)/(k-1)=1+2/(k-1)\in[1,3]$. The function $v\mapsto\ln(2^{v/2}/v)$ is convex, so its maximum on $[1,3]$ is at an endpoint. Thus
		\[
		(k-1)\frac{2^{\frac{k+1}{2(k-1)}}}{k+1}
		=\frac{2^{v/2}}v\le\sqrt2.
		\]
		For the last term, as $d-1\ge k-1$ and $d-1\le(k-1)\log_{4/3}\frac{n}{k-1}$, we have
		\[
		(k-1)\frac{2^{\frac{d-1}{2(k-1)}}}{d-1}
		\le 2^{\frac12\log_{4/3}\frac{n}{k-1}}.
		\]
		This proves~\eqref{eq:weight-bound}.
		
		Since $\ln2<7/10$ and $\ln(4/3)>2/7$, we have
		\[
		\log_{4/3}2<\frac{49}{20}<\frac52,
		\qquad (\log_{4/3}2)\left(1+\frac12\log_{4/3}2\right)<\frac{11}{2}.
		\]
		It follows from~\eqref{eq:d-explicit} and~\eqref{eq:weight-bound} that the two terms in~\eqref{eq:s-threshold} satisfy
		\[
		d+2k-3 \le (k-1)\log_{4/3}\frac{n}{k-1}+2(k-1)<\frac52(k-1)\log_2\frac{n}{k-1}+2(k-1),
		\]
		and
		\begin{align*}
			\left\lfloor(k-1)\log_{4/3}(n\omega)\right\rfloor+1
			&\le (k-1)\left(1+\frac12\log_{4/3}2\right)
			\log_{4/3}\frac{n}{k-1}+1\\
			&<\frac{11}{2}(k-1)\log_2\frac{n}{k-1}+1. 
		\end{align*}
		We have $\log_2\frac{n}{k-1}>1$ and $(k-1)\log_2\frac{n}{k-1}\ge2$: when $k=2$ this follows from $n\ge4$, and when $k\ge3$ it follows from $\log_2\frac{n}{k-1}>1$. Therefore, both terms in~\eqref{eq:s-threshold} are strictly less than $6(k-1)\log_2\frac{n}{k-1}$. Thus, Theorem~\ref{thm:quantitative} gives
		\[
		f(k,n)\le\min\left\{n-1,\left\lfloor6(k-1)\log_2\frac{n}{k-1}\right\rfloor\right\},
		\]
		which completes the proof. 
	\end{proof}
	
	We first give a linear lower bound within the normalization $n\ge2k$.
	
	\section{Concluding remarks}\label{sec:conclusion}
	The main conclusion is $f(k,n)=O(k\log(n/k))$. The result retains a logarithmic dependence on $n$ when $k$ is fixed. Removing that dependence would settle Thomassen's conjecture.
	
	The potential links random alignment to separator descent because losing at most $k-1$ attachment edges multiplies a summand by at most two. The contribution of a non-singleton part is controlled by a boundary matching, while a singleton is controlled by its full star and can be pruned when its degree is too small. Bipartiteness has a second role after the colouring step: it supplies both the edge bound in Lemma~\ref{lem:barrier} and the component-order bound in Lemma~\ref{lem:component-size}.
	
	The present proof still requires
	\[
	\frac nd(3/4)^{\frac{d}{k-1}}<1
	\]
	to control the possible number of non-singleton parts in the initial potential. The size barrier also makes the singleton weight depend on $d$. These are conditions of this argument; they do not establish a limitation on the possible true value of $f(k,n)$.
	
	One possible direction is to use the ambient connectivity again during separator descent, obtaining fresh attachment edges after losses have occurred. Such edges would have to respect the bipartition already chosen. Another is to exploit additional structure of a minimum admissible partition to control its aggregate attachment behaviour. 
	
	Finally, the proof is existential. Choosing a minimum admissible partition does not by itself give a polynomial-time algorithm for finding the spanning bipartite $k$-connected subgraph.

\end{document}